\documentclass[12pt]{amsart}
\usepackage[dvips]{graphicx}
\usepackage{amsmath,graphics}
\usepackage{amsfonts,amssymb}
\usepackage{xypic}
\usepackage{comment}
\specialcomment{proofc}{}{}
\includecomment{proofc}
\theoremstyle{plain}
\newtheorem*{theorem*}{Theorem}
\newtheorem*{lemma*} {Lemma}
\newtheorem*{corollary*} {Corollary}
\newtheorem*{proposition*}{Proposition}
\newtheorem*{conjecture*}{Conjecture}
\newtheorem{theorem}{Theorem}[section]
\newtheorem{lemma}[theorem]{Lemma}
\newtheorem*{theorem1*}{Theorem 1}
\newtheorem*{theorem2*}{Theorem 2}
\newtheorem*{theorem3*}{Theorem 3}
\newtheorem{corollary}[theorem]{Corollary}
\newtheorem{proposition}[theorem]{Proposition}

\theoremstyle{remark}
\newtheorem*{remark}{Remark}

\newtheorem{example*}{Example}
\newtheorem*{claim}{Claim}

\theoremstyle{definition}

\def\gl{\mbox{GL}} \def\Q{\Bbb{Q}}  \def\Z{\Bbb{Z}} \def\R{\Bbb{R}} \def\C{\Bbb{C}}
\def\N{\Bbb{N}}    
 \def\a{\alpha}   \def\bp{\begin{pmatrix}}
\def\sm{\setminus} \def\ep{\end{pmatrix}} \def\bn{\begin{enumerate}} 
 \def\rank{\mbox{rank}}  \def\en{\end{enumerate}}
\def\ba{\begin{array}} \def\ea{\end{array}} 
 \def\S{\Sigma} \def\s{\sigma} \def\a{\alpha}  \def\ti{\tilde}

\def\ker{\mbox{Ker}}\def\be{\begin{equation}} \def\ee{\end{equation}} 
   
 \def\hom{\mbox{Hom}}  
   
   \def\ct{\C[t^{\pm 1}]}

\begin{document}
\title[The Thurston norm and twisted Alexander polynomials]{The Thurston norm and twisted Alexander polynomials}
\author{Stefan Friedl}
\address{Mathematisches Institut\\ Universit\"at zu K\"oln\\   Germany}
\email{sfriedl@gmail.com}

\author{Stefano Vidussi}
\address{Department of Mathematics, University of California,
Riverside, CA 92521, USA} \email{svidussi@math.ucr.edu} \thanks{S. Vidussi was partially supported by NSF grant
DMS-0906281.}
\date{\today}


\begin{abstract}
 Using recent results of Agol, Przytycki-Wise and Wise we  show  that twisted Alexander polynomials detect the Thurston norm  of any  irreducible 3-manifold which is not a closed graph manifold.
\end{abstract}

\maketitle

\section{Introduction and main results}

Let $N$ be a 3-manifold.  (Throughout the paper, unless otherwise stated, we will assume that all
3-manifolds are  orientable, connected and that they have either empty or toroidal boundary.)
Given a surface $\S$ with connected components $\S_1,\dots,\S_k$  its complexity
is defined to be
\[ \chi_-(\S)=\sum_{i=1}^k \max\{-\chi(\S_i),0\}.\]
Given a 3-manifold $N$ and $\phi \in H^1(N;\Z)$ the Thurston norm is defined as
\[ x_N(\phi):= \min\{\chi_-(\S)\, |\, \S \subset N\mbox{ properly embedded and dual to }\phi\}.\]
 Thurston
\cite{Th86} showed that  $x_N$ is a seminorm on $H^1(N;\Z)$.
We  say that a class  $\phi \in H^1(N;\Z)=\hom(\pi_1(N),\Z)$ is \emph{fibered}  if  there exists a fibration $p\colon N\to S^1$ such that
$\phi=p_*:\pi_1(N)\to \Z$. We refer to Section \ref{section:thurstonnorm}
for more information on the Thurston norm and fibered classes.

Given a 3-manifold $N$, a class $\phi \in H^1(N;\Z)$, and a representation $\a\colon \pi_1(N)\to \gl(k,\C)$
we denote by $\tau(N,\phi,\a)\in \C(t)$ the corresponding twisted Reidemeister torsion, whose definition we summarize in Section \ref{section:twialex}.
This invariant was first introduced by Lin \cite{Lin01} and Wada \cite{Wa94} using slightly different normalizations. In the literature $\tau(N,\phi,\a)$ is also sometimes referred to as the twisted Alexander polynomial of $(N,\phi,\a)$.
Our approach in defining $\tau(N,\phi,\a)$ follows the point of view taken in \cite{KL99} and \cite{FK06}.

Given $p(t)\ne 0 \in \ct$ we can write $p(t)=\sum_{i=r}^sa_it^i$ with $a_r\ne 0$ and $a_s\ne 0$ and we define
\[ \deg(p(t)):=s-r.\]
We extend this definition to $\deg(0):=0$. Furthermore, given $f(t)\in \C(t)$ we can write $f(t)=\frac{p(t)}{q(t)}$ with $p(t),q(t)\in \ct$, and we define
\[ \deg(f(t)):=\max\{0,\deg(p(t))-\deg(q(t))\}.\]
Note that the degree of $f(t)$ is well--defined.

In \cite[Theorems~1.1~and~1.2]{FK06} the first author and Taehee Kim proved the following theorem:

\begin{theorem}\label{thm:fk06}
Let $N$ be a 3-manifold, $\phi \in H^1(N;\Z)$ non--zero and $\a\colon \pi_1(N)\to \gl(k,\C)$ a representation, then
\be \label{equ:xy} \frac{1}{k} \deg(\tau(N,\phi,\a))\leq x_N(\phi).\ee
Furthermore equality holds if $\phi$ is fibered.
\end{theorem}

 It is a natural question to ask whether there exists a representation such that (\ref{equ:xy}) becomes an equality for all $\phi$.
Computational evidence towards an affirmative answer was given in \cite{FK06,FK08} and \cite{DFJ11}.
Using recent work of Agol \cite{Ag08,Ag12}, Przytycki-Wise \cite{PW12} and Wise \cite{Wi12} (see Section \ref{section:apw} for details)
we can now prove that that this is indeed the case for most 3-manifolds:

\begin{theorem} \label{main} \label{mainthm} Let $N$ be an irreducible 3-manifold which is not a closed graph manifold.
Then there exists a unitary representation $\a\colon \pi_1(N)\to U(k)$ such that
\[ \frac{1}{k} \deg(\tau(N,\phi,\a))=x_N(\phi) \mbox{ for any }\phi\in H^1(N;\Z)\sm \{0\}.\]
Furthermore $\a$ can be chosen to factor through a finite group.
  \end{theorem}

\begin{remark}
\bn
\item
In Corollary \ref{cor:twinorm} we will prove a closely related result which shows that under the same assumption on $N$ there exists a unitary representation such that the
corresponding twisted Alexander norm introduced
in \cite{FK08} (which generalizes work of McMullen \cite{Mc02} and Turaev \cite{Tu02a}) also detects the Thurston norm.
\item
It was shown by many authors (see \cite{Ch03,GM03,GKM05,FK06,Pa07,Kiy08,FV10}), at different levels of generality,  that twisted Alexander polynomials give obstructions
to a 3-manifold being fibered. In \cite{FV11a} (see also \cite{FV11b}) the authors showed that in fact twisted Alexander polynomials
detect fibered 3-manifolds. The fiberedness criterion of \cite{FV11a} can be greatly strengthened using the recent work of Wise.
We refer to \cite{FV11c} for details.
\item  Some of the ideas of this paper were also used by the authors in \cite{FV11d} to study genus minimizing surfaces in $S^1$--bundles over closed 3-manifolds.
\en
\end{remark}

We conclude with a quick summary of two applications to our results:
\bn
\item Let $N$ be an irreducible 3-manifold which is not a closed graph manifold. Theorem \ref{mainthm} gives a new algorithm for  determining the Thurston norm.
We refer to Section \ref{section:appl}  for details.
\item In \cite{FSW12} we will use Theorem \ref{mainthm} to study `splittings of knot groups
along minimal edge groups' and we will furthermore use Theorem \ref{mainthm} to relate the Thurston norm to the Turaev norm (see \cite{Tu02b}) on the first cohomology of a 2--complex.
\en

\subsection*{Convention.} Unless specified otherwise, all groups are assumed to be finitely generated.
Furthermore we allow norms to be degenerate, i.e. we refer to seminorms as norms.
By a free abelian group we always mean a non--trivial free abelian group.

\subsection*{Acknowledgment.} We are grateful to Genevieve Walsh for a helpful conversation.

\section{The Thurston norm}\label{section:thurstonnorm}

Let $N$ be a 3-manifold and let $\phi\in H^1(N;\Z)$.
It is well--known that  any class in $H^1(N;\Z)$ is dual to a properly embedded surface.
Now recall that the   Thurston norm of
$\phi$ is defined as
\[ x_N(\phi):= \min\{\chi_-(\S)\, |\, \S \subset N\mbox{ properly embedded and dual to }\phi\}.\]
 Thurston
\cite{Th86} showed that  $x_N$ is a seminorm on $H^1(N;\Z)$ which thus
can be extended to a seminorm on $H^1(N;\R)$ which we also denote by $x_N$. Thurston furthermore showed that  the Thurston norm ball
\[ B(N):=\{ \phi \in H^1(N;\R)\, | \, x_N(\phi)\leq 1 \}\]
 is a (possibly non--compact) finite convex polytope.

We also recall that  an integral class $\phi \in H^1(N;\Z)=\hom(\pi_1(N),\Z)$ is called fibered  if  there exists a fibration $p\colon N\to S^1$ such that
$\phi=p_*:\pi_1(N)\to \Z$. More generally, we say that a class $\phi \in H^1(N;\R)$ is  \emph{fibered}  if $\phi$ can be represented by a nowhere vanishing closed 1--form. It is well--known (see e.g. \cite{Ti70}) that for integral classes the two notions of being fibered coincide.

Thurston \cite{Th86} showed that
there exist open top--dimensional faces $F_1,\dots,F_r$ of $B(N)$
such that the set of fibered classes in $H^1(N;\R)$
equals the union of the open cones on $F_1,\dots,F_r$.
The faces $F_1,\dots,F_r$ are referred to as \emph{fibered faces} of $B(N)$.

Finally let $p\colon M\to N$ be a finite cover of degree $k$ and let $\phi\in H^1(N;\R)$.
Then
\be \label{equ:tnfinitecover} x_{M}(p^*\phi)=k\cdot x_N(\phi),\ee
furthermore
$ \phi$ is fibered if and only if $p^*\phi$ is fibered.
We refer to  \cite[Corollary~6.18]{Ga83} for a proof of (\ref{equ:tnfinitecover}), while
the statement regarding finite covers can be proved using Stallings' theorem
\cite{St62}.

\section{Definition and basic properties of twisted Alexander polynomials} \label{section:definition}

\subsection{Twisted Alexander polynomials}\label{section:twialex}

Let $N$ be a 3-manifold and let  $\a\colon \pi_1(N)\to \gl(k,\C)$ be a representation. Furthermore let $\psi\colon \pi_1(N)\to F$ be a rationally surjective  homomorphism to a free abelian group $F$. (Here by rationally surjective we mean that $\psi$ has finite cokernel.)
 We get a tensor representation
\[ \ba{rcl} \a\otimes \psi\colon  \pi_1(N) &\to & \gl(k,\C[F]) \\
g&\mapsto & \a(g)\cdot \psi(g).\ea \]
We denote the universal cover of $N$ by $\ti{N}$.
Note that there exists a canonical left $\pi_1(N)$--action on the universal cover $\ti{N}$ given by deck transformations. We
consider the cellular chain complex $C_*(\ti{N};\Z)$ as a right $\Z[\pi_1(N)]$-module by defining $\s \cdot
g:=g^{-1}\s$ for a chain $\s$. The representation $\a\otimes \psi$  gives rise to  a left action of $\pi_1(N)$ on $\C[F]^k$. We can therefore consider
the  $\C[F]$--complex
\[ C_*^{\psi\otimes \a}(N;\C[F]^k):=C_*(\tilde{N};\Z)\otimes_{\Z[\pi_1(N)]}\C[F]^k.\]
and its homology modules
$$
H_i^{\psi\otimes \a}(N;\C[F]^k) := H_i(C_*(\tilde{N};\Z)\otimes_{\Z[\pi_1(N)]}\C[F]^k).
$$
 Since $N$ is compact and since $\C[F]$
is Noetherian these modules are finitely presented over $\C[F]$.
We now define the \emph{$i$--th twisted Alexander polynomial of $(N,\psi,\a)$} to be the order of $H_i(N;\C[F]^k)$ (see \cite{FV10} and \cite{Tu01} for details). We will denote it as $\Delta_{N,\psi,i}^\a\in \C[F]$.
Throughout this paper we often write
 $\Delta^\a_{N,\psi}$ instead of  $\Delta^{\a}_{N,\psi,1}$.

Note that $\Delta_{N,\psi,i}^\a\in \C[F]$  is well-defined up to multiplication by a unit in $\C[F]$,
i.e. up to an element  of the form $rf$ with $r\in \C\sm\{0\}$ and $f\in F$.
In the following, we denote by $\C(F)$ the quotient field of $\C[F]$, and given $p,q\in \C(F)$ we write
\[ p\doteq q \]
if $p$ and $q$ agree up to multiplication by  an element  of the form $rf$ with $r\in \C\sm\{0\}$ and $f\in F$.

\subsection{Definition of $\tau(N,\psi,\a)$}

The following is a mild extension of \cite[Proposition~2.5]{FK06} and \cite[Lemmas~6.2~and~6.3]{FK08}.
Most of the ideas go back to work of Turaev (cf. e.g. \cite{Tu86} and \cite{Tu01}).

\begin{proposition}\label{prop:deltatau}\label{prop:taudelta}
Let $N$ be a 3-manifold,  $\psi\colon \pi_1(N)\to F$ a rationally surjective  homomorphism to a free abelian group $F$
and $\a\colon \pi_1(N)\to \gl(k,\C)$  a  representation. Then the following hold:
\bn
\item $\Delta_{N,\psi,0}^{\a } \ne 0$.
\item If $\Delta_{N,\psi,1}^{\a } \ne 0$, then $\Delta_{N,\psi,2}^{\a } \ne 0$.
\item If $\rank(F)>1$, then $\Delta^{{\a }}_{N,\psi,0} \doteq 1$.
\item If  $\rank(F)>1$ and if $\Delta_{N,\psi,1}^{\a } \ne 0$, then $\Delta^{{\a }}_{N,\psi,2}\doteq 1$.
\en
\end{proposition}

 If $\Delta_{N,\psi,1}^{\a }\ne 0$, then we define
\[ \tau(N,\psi,{\a } ):=\prod_{i=0}^2 \big(\Delta^{{\a }}_{N,\psi,i}\big)^{(-1)^{i+1}} \in \C(F).\]
If   $\Delta_{N,\psi,1}^{\a }= 0$, then we define  $\tau(N,\psi,\a):=0$.

\begin{remark}
\bn
\item Note that it is an immediate consequence of Proposition \ref{prop:deltatau} that $\tau(N,{ \psi},\a )$ lies in $\C[F]$,
if $\rank(F)>1$.
\item Throughout this paper we identify the complex group ring of $\Z$ with $\ct$, i.e. we identify $\ct=\C[\Z]$ and $\C(t)=\C(\Z)$.
In particular if $\psi\colon \pi_1(N)\to \Z$ is rationally surjective, then we view $\tau(N,\psi,\a)$ as an element in $\C(t)$.
\item
The invariant $\tau(N,\psi,\a)$ can be viewed as a twisted Reidemeister torsion.
We refer to \cite{Mi66,Tu01,Nic03} for background on Reidemeister torsion,
and we refer to \cite{Kio96,KL99,FV10} and   \cite[Theorem~6.7]{FK08} for twisted Reidemeister torsion and its relation to twisted Alexander polynomials. We will not make use of this point of view.
\item We drop $\psi$ from the notation if $\psi$ is the projection $\pi_1(N)\to H_1(N;\Z)/\mbox{torsion}$, furthermore we drop $\a$ from the notation if $\a$ is the trivial one--dimensional representation.
\en
\end{remark}

\subsection{Tensoring with one-dimensional representations}

Let $N$ be a 3-manifold, let
$\psi\colon \pi_1(N)\to F$ be a rationally surjective  homomorphism to a free abelian group $F$
and let $\rho\colon F\to U(1)$ be a character.
We will denote the character $\rho\circ \psi\colon \pi_1(N)\to U(1)$ by $\rho$ as well.
Note that $\rho$  gives rise to a ring automorphism $\C[F]\to \C[F]$ induced by $f\mapsto \rho(f) \cdot f, f\in F$.
We will denote this ring automorphism by $\rho$ as well.

The following lemma is now a straightforward  consequence of the definitions:

\begin{lemma}\label{lem:twist}
For any $i$ we have
\[ \Delta_{N,\psi,i}^{\rho} = \rho(\Delta_{N,\psi,i})\in \C[F].\]
\end{lemma}

\subsection{Change of variables}

Let $N$ be a 3-manifold. We write $F=H_1(N;\Z)/\mbox{torsion}$.
Let $\a\colon \pi_1(N)\to \gl(k,\C)$ be a  representation.
Furthermore let $\phi\in H^1(N;\Z)=\hom(F,\Z)$.
We denote the induced ring homomorphism $\C[H]\to \C[\Z]=\ct$ by $\phi$ as well.
Let
\[ S=\{ f \in \C[F] \, |\, \varphi(\C[F])\ne 0 \in \ct\}.\]
Note that $\phi$ induces a homomorphism $\C[F]S^{-1}\to \C(t)$ which we also denote by $\phi$.

The following is  a slight generalization of \cite[Theorem~6.6]{FK08}, which in turn builds on ideas of Turaev
(cf. \cite{Tu86} and \cite{Tu01}).

\begin{proposition}\label{prop:changevartau}
We have  $\tau(N,\a )\in \C[F]S^{-1}$, and
\[ \tau(N, \phi,\a) \doteq \phi(\tau(N,\a)).\]
\end{proposition}

\subsection{Induced representations}

Let $\pi$ be a group and let $\pi'\subset \pi$ be a normal subgroup of index $k$.  Let $\rho\colon \pi' \to U(1)$ be a character.
We  consider the action of $\pi$ given by left multiplication on the tensor product
\[ \C[\pi]\otimes_{\C[\pi']} \C,\]
where $\pi'$ acts on $\C$ via the character $\rho$.
Note that $ \C[\pi]\otimes_{\C[\pi']} \C$ is a complex $k$--dimensional vector space and we thus obtain a representation
$\a\colon \pi\to \gl(k,\C)$. This representation is called an \emph{extended character}.
If $\rho$ factors through a finite group, then we refer to $\a$ as an \emph{extended finite character}.
For future reference we record the following elementary fact:

\begin{lemma}\label{lem:induced}
\bn
\item An extended character  is a unitary representation.
\item An extended finite character factors through a finite group.
\en
\end{lemma}

We conclude this section with the following well--known lemma:

\begin{lemma}\label{lem:shapironorm}
Let $N$ be a 3-manifold  and let $p\colon M\to N$ be a finite regular covering map
of index $k$. Let $\rho\colon \pi_1(M) \to U(1)$ be a character
and denote by $\a\colon \pi_1(N)\to U(k)$ the extended character.
Let $\psi\colon \pi_1(N)\to F$ be a rationally surjective  homomorphism to a free abelian group $F$. Then
\[ \tau(M,\psi\circ p_*,\rho)=\tau(N,\psi,\a).\]
\end{lemma}

\begin{proof}
It is an immediate consequence of Shapiro's lemma (cf. \cite[Proposition~6.2]{Br94} and also \cite{FV08})
that for any $i$ the following equality holds:
\[ \Delta_{M,\psi\circ p_*,i}^{\rho}=\Delta_{N,\psi,i}^{\a}\in \C[F].\]
The lemma now follows from the definitions.
\end{proof}

\section{The twisted Alexander norm} \label{sec:twisted alexandernorm}

\subsection{Definition of the twisted Alexander norm}
We now recall the definition of twisted Alexander norms introduced in \cite{FK08}.
 Let $N$ be a 3-manifold with $b_1(N)>1$. We write $F:=H_1(N;\Z)/\mbox{torsion}$.
 Let $\a\colon \pi_1(N)\to \gl(k,\C)$
 be a representation. We will now define a seminorm $y^\a_N$ on  $H^1(N;\R)=\hom(F,\R)$ using $\tau(N,\a)$.

 If $\tau(N,\a)=0$ then we set
$y^\a_{N}(\phi)=0$ for any $\phi\in H^1(N;\R)$.
Now suppose that  $\tau(N,\a)\ne 0$.
Since $b_1(N)=\operatorname{rank}(F)>1$ it follows from Proposition \ref{prop:taudelta}
that $\tau(N,\a)\in \C[F]$. We can therefore  write
$\tau(N,\a)=\sum_{f\in F}a_ff$ and we define
\[ \ba{rcl} y_{N}^\a\colon \hom(F,\R)&\to & \R_{\geq 0} \\
\phi&\mapsto & \max\{ \phi(f_1)-\phi(f_2)\,|\,f_1,f_2\in F\mbox{ with } a_{f_1}\ne 0 \mbox{ and } a_{f_2}\ne 0\}.\ea \]
It is clear that $y_{N}^\a$ defines a norm on $H^1(N;\R)=\hom(F,\R)$ and we refer to $y^\a_{N}$ as the
\emph{twisted Alexander norm of $(N,\a)$}.
If $\a$ is the trivial one--dimensional representation, then we drop $\a$ from the notation.

\begin{remark}
\bn
\item With the above conventions the seminorm $y_N$ is just the ordinary Alexander norm introduced by McMullen
\cite{Mc02}.
\item
 Twisted Alexander norms corresponding to abelian one--dimensional representations were first considered by  Turaev
\cite{Tu02a}, twisted Alexander norms for arbitrary representations were introduced in \cite{FK08}.
\en
\end{remark}

\subsection{Lower bounds on the Thurston norm and fibered classes}

In this section we recall results relating the Thurston norm of a 3-manifold to  the twisted Alexander norms. We start out with the following lemma.

\begin{lemma}\label{lem:taunorm}
Let $N$ be a 3-manifold with $b_1(N)>1$ and  let $\a\colon \pi_1(N)\to \gl(k,\C)$ be a representation. Given any  $\phi\in H^1(N;\Z)$ we have
\[ \deg(\tau(N,\phi,\a)) \leq y_{N}^\a(\phi),\]
furthermore  equality holds for all $\phi$ outside a finite collection of hyperplanes in $H^1(N;\Q)$.
\end{lemma}

The proof is essentially given in \cite{Mc02} and we provide it for the reader's convenience
as it  is helpful in understanding the proof of Theorem \ref{mainthm}.

\begin{proof}
We write $F:=H_1(N;\Z)/\mbox{torsion}$.
By Proposition \ref{prop:taudelta} we can   write
$\tau(N,\a)=\sum_{f\in A}a_ff$ with $A\subset F$ and $a_f\ne 0$ for all $f\in A$.
 It follows from Proposition   \ref{prop:changevartau} that
\[ \tau(N,\phi,\a)=\phi(\tau(N,\a))=\sum_{f\in F}a_ft^{\phi(f)}.\]
It follows that
\[ \deg(\tau(N,\phi,\a))=\deg\left(\sum_{f\in F}a_ft^{\phi(f)}\right).\]
It is clear that
\[ \deg(\sum_{f\in F}a_ft^{\phi(f)})\leq  \max\{ \phi(f_1)-\phi(f_2)\,|\,f_1,f_2\in A\},\]
and that equality holds
unless there exist $f_1,f_2\in A$ with $\phi(f_1)=\phi(f_2)$.
The lemma now follows immediately.
\end{proof}

The following theorem is now a consequence of the above lemma, Theorem \ref{thm:fk06}
and the fact that  norms are continuous. The theorem was first proved
in \cite[Theorem~3.1]{FK08}.

\begin{theorem}\label{thm:fk08}
Let $N$ be a 3-manifold with $b_1(N)>1$ and  let $\a\colon \pi_1(N)\to \gl(k,\C)$ be a representation.
 Then for any $\phi\in H^1(N;\R)$ we have
\[ \frac1k y_{N}^\a(\phi) \le x_N(\phi)   \]
and equality holds if $\phi$  is a fibered class.
\end{theorem}

\noindent The first part of  Theorem \ref{thm:fk08} generalizes McMullen's theorem  \cite{Mc02}. Turaev \cite{Tu02a} proved this theorem in the
special case of abelian representations.

\section{Proof of Theorem \ref{mainthm}}

\subsection{Agol's virtual fibering theorem} \label{section:agol}

We say that $\phi \in H^1(N;\R)$ is \emph{quasi--fibered}  if there exists a fibered face $F$ of the Thurston norm ball such that $\phi$ lies in a cone on the closure of $F$.  Put differently,  $\phi$ is quasi--fibered if and only if any neighborhood of $\phi$ in $H^1(N;\R)$ contains a fibered class.
Note that in particular fibered classes are quasi--fibered.

We recall from \cite{Ag08} the following definition:
A group $\pi$ is  \emph{residually finite rationally solvable (RFRS)} if there
exists a filtration  of groups $\pi=\pi_0\supset \pi_1 \supset \pi_2\dots $
such that the following hold:
\newcounter{itemcountern}
\begin{list}
{{(\arabic{itemcountern})}}
{\usecounter{itemcountern}\leftmargin=2em}
\item $\bigcap_i \pi_i=\{1\}$;
\item   $\pi_i$ is a normal, finite-index subgroup of  $\pi$ for any $i$;
\item for any $i$ the map $\pi_i\to \pi_i/\pi_{i+1}$ factors through $\pi_i\to H_1(\pi_i;\Z)/\mbox{torsion}$.
\end{list}
We refer to \cite{Ag08} for details and more information on RFRS groups.
 A group is \textit{virtually} RFRS if it admits a finite index subgroup that is RFRS.

We can now formulate Agol's virtual fibering theorem (see \cite{Ag08}) which is one of the key ingredients in our proof of Theorem \ref{mainthm}.

\begin{theorem} \textbf{\emph{(Agol)}}\label{thm:agol}
Let $N$ be an irreducible 3-manifold such that $\pi_1(N)$ is virtually RFRS. Then given any  $\phi \in H^1(N;\R)$ there exists a finite regular cover $p\colon M\to N$, such that $p^*\phi\in H^1(M;\R)$ is quasi--fibered.
\end{theorem}

The following is a well--known consequence of Agol's virtual fibering theorem:

\begin{corollary}\label{cor:agol}\label{cor:quasifi}
Let $N$ be an irreducible $3$--manifold with virtually RFRS fundamental group.
 Then there exists a finite regular cover $p \colon  {M} \to N$  such that for every nontrivial class $\phi \in H^1(N;\R)$, the class $p^* \phi \in H^1({M};\R)$  is quasi-fibered. \end{corollary}

\begin{proof}
Let   $\phi \in H^1(N;\R)$ be a class contained in the cone over a top--dimensional open face of $B(N)$.
By Agol's virtual fibering theorem there exists a finite regular cover  $p \colon  {M} \to N$  such that $p^{*} \phi$ is quasi--fibered.
By (\ref{equ:tnfinitecover}) the pull back  map $p^{*} \colon  H^1(N;\R) \to H^1({M};\R)$ is, up to scale, a monomorphism of normed vector spaces when we endow these spaces with their respective Thurston norm. It follows that the pull back under $p$ of any class in the \textit{closure} of the open cone (in $H^1(N;\R)$) determined by $\phi$ will be quasi--fibered in ${M}$. For the same reason if a class lies in the closure of a fibered cone, its pull back under further finite covers will enjoy the same property (since pullbacks of  fibrations are fibrations). Recall now that the Thurston norm ball of a $3$--manifold is a finite, convex polyhedron, in particular it has finitely many top--dimensional open faces. By picking one class in the cone above each of these faces, and repeatedly applying Agol's theorem to the (transfer of) each such class, we obtain after finitely many steps a finite regular cover, that we will denote as well by $p \colon  {M} \to N$,
such that $p^*\phi$ is quasi--fibered for any $\phi \in H^1(N;\R)$.
By going to a further cover, if necessary, we can arrange that $M$ is in fact a finite regular cover of $N$.
\end{proof}

\subsection{The results of Agol, Liu, Przytycki-Wise and Wise} \label{section:apw}

The `virtually RFRS' condition in Theorem \ref{thm:agol} might at a first glance look very restrictive.
It is thus amazing that over the last few years it was shown that `most' 3-manifold groups are in fact virtually RFRS.

The key in proving that  3-manifold groups are virtually RFRS is the notion of a `virtually special' group
introduced by Haglund and Wise \cite{HW08}.
The precise definition of `virtually special' is of no concern to us. The only thing we need is the following
theorem of Haglund and Wise \cite{HW08} and Agol \cite[Theorem~2.2]{Ag08}:

\begin{theorem} \label{thm:rfrs}
Let $\pi$ be a group which is virtually special, then $\pi$ is virtually RFRS.
\end{theorem}

We can now formulate  the following theorem  which was announced by
Wise \cite{Wi09} in 2009, with the details of the proof being provided in  \cite{Wi12} (see also  \cite{Wi11}).

\begin{theorem} \textbf{\emph{(Wise)}}\label{thm:wiseintro}\label{thm:wise2}\label{thm:wise}
Let $N$ be a hyperbolic 3-manifold such that one of the following holds:
\bn
\item $b_1(N)>1$, or
\item $b_1(N)=1$ and $N$ is not fibered,
\item $N$ has non-trivial boundary,
\en
then $\pi_1(N)$ is virtually special.
\end{theorem}

\begin{remark}
We give now precise references for the theorem:
\bn
\item Let $N$ be a closed hyperbolic 3--manifold which admits a geometrically finite surface.
 Theorem 14.1 of \cite{Wi12} asserts  that $\pi_1(N)$ is `virtually special', which by work of
 implies that $\pi_1(N)$ is virtually RFRS. Thurston and Bonahon  \cite{Bo86} showed that an incompressible connected surface $\Sigma\subset N$ either $\Sigma$ lifts to a surface fiber in a finite cover, or $\Sigma$ is geometrically finite.
It now follows from basic arguments that a closed hyperbolic 3-manifold with either
\bn
\item $b_1(N)>1$, or
\item $b_1(N)=1$ and $N$ not fibered,
\en
 admits a geometrically finite surface. By the above this implies that the fundamental group of such a hyperbolic 3-manifold is virtually RFRS.
\item  If $N$ has boundary, then the statement of Theorem \ref{thm:wise2} is precisely Theorem 16.28 together with Corollary 14.16 of \cite{Wi12}.
\item
 Agol \cite{Ag12}, building on the Surface Subgroup Theorem of Kahn-Markovic \cite{KM12} and on work of Wise \cite{Wi12}
 showed that in fact  the fundamental group of \textit{any} closed hyperbolic 3-manifold is virtually RFRS. We will not make use of this result
 since in our situation it suffices to consider hyperbolic 3-manifolds which satisfy the conditions of Theorem \ref{thm:wise2}.
 Also note that Agol's result does not cover the case of hyperbolic 3-manifolds with non-trivial toroidal boundary.
 \en
 \end{remark}

The following theorem was proved by Liu \cite[Theorem~1.1]{Liu11}.

\begin{theorem} \textbf{\emph{(Liu)}}\label{thm:liu}
Let $N$ be a graph manifold which supports a non-positively curved metric.
Then $\pi_1(N)$ is virtually special.
\end{theorem}

Note that by \cite{Le95} any graph manifold with non-trivial boundary is non-positively curved.
For such graph manifolds the theorem was also proved by Przytycki-Wise \cite{PW11}.

The following theorem is also due to Przytycki-Wise.

\begin{theorem} \textbf{\emph{(Przytycki-Wise)}}\label{thm:pw12}
Let $N$ be a 3-manifold which admits a nontrivial JSJ decomposition with at least one hyperbolic piece.
Then $\pi_1(N)$ is virtually special.
\end{theorem}

\subsection{Properties of the twisted Alexander norm}

Let $F$ be an free abelian  group
and let $\rho\colon F\to U(1)$ be a character.
Note that $\rho$ gives rise to a ring homomorphism $\C[F]\to \C$ which we also denote by $\rho$.

\begin{lemma}\label{lem:nonzerochar}\label{lem:nontrivchar}
Let $F$ be a free abelian  group and let $\{p_i =  \sum_{j=1}^{d_i} a_{ij}f_{ij} \in \C[F], i = 1,\dots,l\}$  be a collection of non--zero polynomials.
Then there exists a character $\rho\colon F\to U(1) \subset \C$ which factors through a finite group such that all $\sum_{j=1}^{d_i} a_{ij}\rho(f_{ij}) \in \C$ are simultaneously non--zero.
\end{lemma}

\begin{proof}
Given $i\in \{1,\dots,l\}$  it is clear that
\[ V_i:=\{\kappa \in \hom(F,\Q)\,|\, \kappa(f_{ij})=\kappa(f_{ik})\mbox{ for some }j\ne k,  a_{ij} a_{ik} \neq 0\}\]
is a finite union of codimension one subspaces of $\hom(F,\Q)$.
We can  now pick an epimorphism $\kappa \colon F\to \Z$ such that $\kappa \not\in V_1\cup \dots \cup V_l$.
For all $z\in U(1) \subset \C$ we now consider the character
\[ \ba{rcl} \kappa_z:F&\to & U(1) \\
f&\mapsto & z^{\kappa(f)}.\ea \]
Since $\kappa \not \in V_1\cup \dots \cup V_l$, the polynomials  $q_i(t) =  \sum_{j=1}^{d_i} a_{ij} t^{\kappa(f_{ij})} \in \ct$ are non--zero. It is clear that there exists a root of unity $z$ with $q_i(z)\ne 0 \in \C$ for $i=1,\dots,l$. In correspondence of such $z \in U(1)$ the character $\rho := \kappa_{z}$ has the desired property.

\end{proof}

\begin{proposition} \label{prop:makeequal}
Let $M$ be a 3-manifold.
Then  there exists a character $\rho\colon \pi_1(M)\to U(1)$ which factors through a finite group such that
\[ y_{M}(\phi)=\deg(\tau(M,\phi,\rho))\]
for any $\phi \in H^1(M;\Z)$.
\end{proposition}

\begin{proof}
We
write $F:=H_1(M;\Z)/\mbox{torsion}$.
If $\operatorname{rank}(F)=1$ or if $\tau(M)=0$, then there is nothing to prove. We thus consider the case $\operatorname{rank}(F)>1$ and $\tau(M)\ne 0$.
By Proposition \ref{prop:taudelta} we have  $\tau(M)=\Delta_M\in \C[F]$.
We now  write
\[ \Delta_M=\sum_{f\in F}a_ff\]
with $a_f\in \C, f\in F$.
Given a subset $A\subset F\otimes \R$ we write
\[ \Delta_A:=\sum_{f\in F\cap A}a_ff\in \C[F].\]
We say $\Delta_A \in \C[F]$ is a \emph{face polynomial of $\Delta_M$} if $A \subset  N(\Delta_{M})$ is a \textit{face} of the Newton polyhedron
of $\Delta_{M}$.

We denote by $\Delta_{A_1},\dots,\Delta_{A_l}$ the face polynomials of $\Delta_M$.
By Lemma \ref{lem:nonzerochar} there exists a character $\rho\colon F\to U(1)$ factoring through a finite group such that \[  \sum_{f \in F \cap A_{1}} a_{f}\rho(f), \dots, \sum_{f \in F \cap A_{l}} a_{f}\rho(f)\] are simultaneously non--zero complex numbers.

\begin{claim}
Let $\phi \in H^1(M;\Z)=\hom(F,\Z)$, then  $y_{M}(\phi)=\deg(\tau(M,\phi,\rho))$.
\end{claim}

First note that $\phi$ extends to a homomorphism $\hom(F\otimes \R,\R)$ which we also denote by $\phi$.
There exist maximal faces $A_t$ and $A_b$ of $N(\Delta_{M})$, such that $\phi$
takes on maximal values on $A_t\subset F\otimes \R$ and minimal values on $A_b\subset F \otimes \R$. Put differently, $\phi$ is constant on $A_t$ and $A_b$, and
\[ \phi(A_b)\leq \phi(f)\leq \phi(A_t)\]
for any $N(\Delta_{M})$ and equalities hold only if  $f\in A_b$ respectively $f\in A_t$.
Recall that \begin{equation} \label{equ:leadingcoeff} \sum_{f \in F \cap A_{b}} a_{f}\rho(f )  \neq 0 \in \C \ \  \mbox{and} \ \  \sum_{f \in F \cap A_{t}} a_{f}\rho(f)  \neq 0 \in \C. \end{equation} At this point, applying
Lemma \ref{lem:twist} and Proposition \ref{prop:changevartau}
we have
\[ \tau(M,\phi,\rho)=\sum_{f\in F} a_{f}\rho(f) t^{\phi(f)}.\]
It follows from (\ref{equ:leadingcoeff}) and the above discussion that
\[ \ba{rcl} \mbox{highest degree term of $\tau(M,\phi,\rho)\in \ct$}&=& (\sum_{f \in F \cap A_{t}} a_{f}\rho(f ))t^{\phi(A_{t})}, \mbox{ and }\\
 \mbox{lowest degree term of $\tau(M,\phi,\rho)\in \ct$}&=& (\sum_{f \in F \cap A_{b}} a_{f}\rho(f ))t^{\phi(A_{b})},\ea \]
hence
\[ \ba{rcl} y_{M}(\phi)&=&\max\{ \phi(f_1)-\phi(f_2)\,|\, f_1,f_2\in F\mbox{ with } a_{f_1}\ne 0\mbox{ and }a_{f_2}\ne 0\} \\
&=& \phi(A_t)-\phi(A_b)\\
&=&\deg(\tau(M,\phi,\rho)).\ea \]
This concludes the proof of the claim.
\end{proof}

\subsection{Proof of Theorem \ref{mainthm}}

The following theorem, together with
Lemma \ref{lem:induced}, clearly implies Theorem \ref{mainthm}.

\begin{theorem} \label{mainthmext}
Let $N$ be an irreducible 3-manifold which is not a closed graph manifold.
Then there exists an extended finite character $\a:\pi_1(N)\to U(k)$ such that
\[ \frac{1}{k} \deg(\tau(N,\phi,\a))=x_N(\phi) \mbox{ for any }\phi\in H^1(N;\Z).\]
  \end{theorem}

\begin{proof}
Let $N$ be an irreducible 3-manifold which is not a closed graph manifold.
If $b_1(N)=0$, then there is nothing to prove.
If  $N$ is a closed hyperbolic fibered 3-manifold with $b_1(N)=1$,
then the trivial representation has the required property.
For all the remaining cases, it follows from Theorems \ref{thm:rfrs}, \ref{thm:wise}, \ref{thm:liu} and \ref{thm:pw12} that $\pi_1(N)$ is virtually RFRS.
By Corollary \ref{cor:quasifi} there
exists a regular finite cover $p\colon M\to N$ such that $p^*\phi$ is quasi--fibered for any $\phi \in H^1(N;\R)$.
We denote by $k$ the order of the cover.
It follows from Theorem \ref{thm:fk08} and from the continuity of $y_{M}$ and $x_{M}$ that
$ y_{M}(\psi)=x_{M}(\psi)$
for any quasi--fibered class $\psi\in H^1(M;\R)$. In particular we obtain  that
\be\label{equ:1} y_{M}(p^*\phi)=x_{M}(p^*\phi) \mbox{ for any } \phi\in H^1(N;\R).\ee

By Proposition  \ref{prop:makeequal}  there exists a character $\rho\colon \pi_1(M)\to U(1)$ which factors through a finite group such that
\be \label{equ:2} y_{M}(\psi)=\deg(\tau(M,\psi,\rho))\ee
for any $\psi \in H^1(M;\Z)$.
We now denote by $\a\colon \pi_1(N)\to U(k)$ the corresponding extended  finite character.

Now let $\phi\in H^1(N;\Z)$.
Lemma
\ref{lem:shapironorm} implies
\be \label{equ:3} \tau(M,p^*\phi,\rho)=\tau(N,\phi,\a)\ee
It  now follows from   Equalities (\ref{equ:tnfinitecover}), (\ref{equ:1}), (\ref{equ:2}) and (\ref{equ:3}) and from Theorem \ref{thm:fk06} that
\[ \ba{rcl}x_N(\phi)&=&\frac{1}{k} x_{M}(p^*\phi)\\[2mm]
&=&\frac{1}{k}y_M(p^*\phi)\\[2mm]
&=&\frac{1}{k}\deg(\tau(M,p_*\phi,\rho))=\frac{1}{k} \deg(\tau(N,\phi,\a))\leq x_N(\phi).\ea\]
We conclude that
\[ \frac{1}{k} \deg(\tau(N,\phi,\a))= x_N(\phi)\mbox{ any $\phi \in H^1(N;\Z)$.}\]
\end{proof}

We can also reinterpret  Theorem \ref{mainthm} in terms of twisted Alexander norms:

\begin{corollary} \label{maincorextnorm}\label{cor:twinorm}
Let $N$ be an irreducible 3-manifold which is not a closed graph manifold.
Then there exists an extended finite character $\a:\pi_1(N)\to U(k)$  such that
\[ \frac{1}{k}y_N^\a(\phi)=x_N(\phi)  \mbox{ for any }\phi\in H^1(N;\R).\]
  \end{corollary}

\begin{proof}
Let $N$ be an irreducible 3-manifold which is not a closed graph manifold.
By Theorem \ref{mainthmext} there   exists an extended finite character  $\a:\pi_1(N)\to U(k)$  such that
\[ \frac{1}{k} \deg(\tau(N,\phi,\a))=x_N(\phi) \mbox{ for any }\phi\in H^1(N;\Z).\]
Lemma \ref{lem:taunorm} and Theorem \ref{thm:fk08} imply that
\[ x_N(\phi)=\frac{1}{k} \deg(\tau(N,\phi,\a))\leq \frac{1}{k}  y_N^\a(\phi)\leq x_N(\phi),\]
hence for any $\phi\in H^1(N;\Z)$ we have
\[ \frac{1}{k}y_N^\a(\phi)=x_N(\phi) .\]
It now follows from the linearity and the continuity of the norms $y_N^\a$ and $x_N$
that this equality holds in fact for all real classes.
\end{proof}

\section{An algorithm for determining the Thurston norm}\label{section:appl}

Let $N$ be an irreducible 3-manifold which is not a closed graph manifold.
We will now show the following:
\bn
\item[(A)]  Given  $\phi\in H^1(N;\Z)$ Theorem \ref{mainthmext} gives rise to an algorithm $A$ which determines the Thurston norm
of $\phi$.
\item[(B)] Corollary \ref{maincorextnorm} gives rise to an algorithm $B$ which determines the Thurston norm of $N$.
\en
Since the former algorithm is much easier to explain we treat it separately, even though of course the second algorithm is stronger than the first algorithm.
We explain the algorithms in a somewhat informal way, we leave it to the reader to formulate a completely formal algorithm.

\subsection{Extended finite characters}

Let $\pi$ be a finitely presented group. We can then systematically go through all homomorphisms from $\pi$ to all permutation groups.
Since every finite group is a subgroup of a permutation group we can thus go through all epimorphisms to finite group.
For each epimorphism $\a\colon \pi\to G$ onto a finite group we can determine $H_1(\ker(\a);\Z)$ using the Reidemeister--Schreier method.
It is now straightforward to see that one can systematically find a sequence of extended finite characters, such that up to conjugation
every finite extended character will eventually appear.

Put differently,  we can inductively define extended finite characters $\a_i$, $i\in \N$,
such that given any extended finite character $\a$ there exists an $i$, such that $\a$ and $\a_i$ are conjugate.

\subsection{Algorithm $A$}
Let $N$ be an irreducible 3-manifold which is not a closed graph manifold and let $\phi\in H^1(N;\Z)$.

The algorithm $A$  consists of two programs $P(\phi)$ and $Q(\phi)$ running at the same time:
\bn
\item For $i=1,2,3,\dots$ Program $P(\phi)$ computes $\tau(N,\phi,\a_i)\in \C(t)$.
Note that $\tau(N,\phi,\a_i)$ can be calculated efficiently using Fox calculus (see e.g. \cite{FK06}).
\item Program $Q(\phi)$  lists all properly embedded surfaces dual to $\phi$ up to isotopy and computes their complexities.
Such a program can for example be written using normal surfaces (see e.g. \cite{CT09}).
\en
It follows from Theorem \ref{mainthmext} that after finitely many steps the lower bound on the Thurston norm coming from $P(\phi)$ agrees with the upper bound on the Thurston norm coming from $Q(\phi)$.

This algorithm sounds inefficient, but the calculations in \cite{FK06} and \cite{DFJ11}
show that in practice twisted Alexander polynomials are very efficient at determining the Thurston norm for a given class $\phi$.

\subsection{Algorithm $B$}

Before we explain the algorithm for determining the Thurston norm ball we state the following elementary lemma:

\begin{lemma}\label{lem:samenorm}
Let $x$ and $y$ be seminorms on a real vector space $V$ with $y\leq x$.
Denote by $B$ and $C$ the norm balls of $x$ and $y$.
Let $\phi_1,\dots,\phi_k$ be a set of non--zero classes in $V$ such that each open cone on a top dimensional face of $C$ contains at least
one $\phi_i$.
If $x(\phi_i)=y(\phi_i)$ for $i=1,\dots,k$, then $x=y$ for all $\phi\in V$.
\end{lemma}

Let $N$ be an irreducible, triangulated 3-manifold which is not a closed graph manifold.
Recall that we denote the Thurston norm ball of $N$ by
\[ B(N):=\{ \phi \in H^1(N;\Q)\,|\, x_N(\phi)\leq 1\}.\]
Given a representation $\a\colon \pi_1(N)\to \gl(k,\C)$ we also write
\[ B(N,\a):=\{ \phi\in H^1(N;\Q)\,|\, \frac{1}{k}y_{N}^\a(\phi)\leq 1\}.\]
Note that Theorem \ref{thm:fk08} says that $B(N,\a)\subset B(N)$.
Also note that $B(N,\a)$ is the dual to the Newton polygon of $\tau(N,\a)$, in particular the vertices of $B(N,\a)$
can be determined easily using $\tau(N,\a)$.

We finally  write $\Phi:=\emptyset \subset H^1(N;\Z)$ and we denote by $z$ the zero norm.
The algorithm now consists of two programs running at the same time:
\bn
\item If $\phi_1,\dots,\phi_k$ denote the elements of $\Phi$, then we apply algorithms $Q(\phi_1),\dots,Q(\phi_k)$ from the previous section and we compute the complexities of the corresponding surfaces.
\item For $i=1,2,\dots$ Program $P$ computes $\tau(N,\a_i)$ and determines $B(N,\a_i)$.
If
\[ B(N,\a_j)\subsetneq B(N,\a_i) \mbox{ for }j=1,\dots,i-1,\]
then
we denote by $z$ the norm $y_{N}^{\a_i}$ and for each open cone on a top dimensional face of $B(N,\a_i)$ we pick a class in $H^1(N;\Z)$ and we denote the resulting set by $\Phi$ and we restart (1).
\en
We terminate the two programs when the complexities of $C(\phi_1),\dots,C(\phi_k)$ agree with $z(\phi_1),\dots,z(\phi_k)$.
It then follows from Lemma \ref{lem:samenorm} that $z$ equals the Thurston norm on $N$.

It remains to show that this algorithm terminates after finitely many steps.
First note that  by  Corollary \ref{maincorextnorm} there exists an $i$ such that
$B(N,\a_i)=B(N)$. After finitely many further steps the first program will find Thurston norm surfaces
representing $\phi_1,\dots,\phi_k$. The program terminates at this point.


\end{document}